\documentclass[11pt]{article}
\usepackage{amsmath}
\usepackage{amsthm}
\usepackage{amsfonts}
\usepackage[latin5]{inputenc}
\usepackage{color}

\begin{document}
\numberwithin{equation}{section}
\newcounter{thmcounter}
\newcounter{Remarkcounter}
\newcounter{Defcounter}
\numberwithin{thmcounter}{section}
\newtheorem{Prop}[thmcounter]{Proposition}
\newtheorem{Corol}[thmcounter]{Corollary}
\newtheorem{theorem}[thmcounter]{Theorem}
\newtheorem{Lemma}[thmcounter]{Lemma}
\theoremstyle{definition}
\newtheorem{Def}[Defcounter]{Definition}
\theoremstyle{remark}
\newtheorem{Remark}[Remarkcounter]{Remark}
\newtheorem{Example}[Remarkcounter]{Example}

\newcommand{\diag}{\mathrm{diag}\ }
\newcommand{\DIV}{\mathrm{div}}
\title{H-hypersurfaces with  3 distinct principal curvatures in the Euclidean spaces}
\author{Nurettin Cenk Turgay\footnote{The final version of this paper is published in `Annali di Matematica Pura ed Applicata, see \cite{TurgayHHypersurface}'} \footnote{Istanbul Technical University, Faculty of Science and Letters, Department of  Mathematics, 34469 Maslak, Istanbul, Turkey} \footnote{E-mail:turgayn@itu.edu.tr, Phone: (+90)533 227 0041 Fax: (+90)212 285 6386} \footnote{Author is supported by Scientific Research Agency of Istanbul Technical University.}}
\date{}

\maketitle
\begin{abstract}
In this paper, we study hypersurfaces of Euclidean spaces with arbitrary dimension. First, we obtain some results on  $\mbox{H}$-hypersurfaces. Then, we give the complete classification of $\mbox{H}$-hypersurfaces with 3 distinct curvatures.  We also give explicit examples.


\end{abstract}

\section{Introduction}\label{SectionIntrod}
Let  $M$ be an $n$-dimensional submanifold of Euclidean $m$-space $\mathbb E^m$ and $x:M\rightarrow\mathbb E^m$ an isometric immersion. $M$ is said to be biharmonic if $x$ satisfies $\Delta^2 x=0$,  where $\Delta$ is the Laplace operator of $M$.  In \cite{ChenOpenProblems,ChenRapor},  Bang-Yen Chen conjectured that  every biharmonic submanifold of a Euclidean space is minimal. This conjecture is supported by all of the results obtained so far (see for example \cite{ChenMunt1998Bih,Dimitric1992,HsanisField}).

On the other hand, $M$ is said to be null 2-type if $x$ can be expressed as $x=x_0+x_1$ for some non-constant vector valued functions $x_0$ and $x_1$ satisfying $\Delta x_0=0$ and $\Delta x_1=\lambda x_1$ for a non-zero constant $\lambda$, \cite{ChenKitap,ChenKitapNew}. Several works on null 2-type surfaces also have been appeared, \cite{Chen1994a,Dimitric2009,Garay1994}.

In particular, there are some results on biharmonic and null 2-type hypersurfaces appeared recently, \cite{ChenGaray,FuYu2014,Garay1994}. For example, in \cite{ChenGaray}, authors obtained results on  $\delta(2)$-ideal null 2-type hypersurfaces. Most recently, Yu Fu has studied biharmonic hypersurfaces in $\mathbb E^5$ with 3 principle curvatures and he has proved that the biharmonic conjecture is true for this case, \cite{FuYu2014}. 

Now, suppose that $M$ is a hypersurface in Euclidean space $\mathbb E^{n+1}$ and let $N$ be its unit normal vector field.  From the definition, one can see that if $M$ is null 2-type or biharmonic, then the equation 
\begin{equation}\nonumber
\Delta^2x=\lambda \Delta x
\end{equation}
is satisfied for a constant $\lambda$. In addition, Beltrami's well known formula  $\Delta x=s_1N$  implies $$\Delta^2 x=\left(\Delta s_1+s_1(s_1^2-2s_2)\right)N+\left(S(\nabla s_1)+\frac {s_1}2\nabla s_1\right),$$ where $S$ is the shape operator  and $s_1$ and $s_2$ denote the first and second mean curvatures of $M$. Therefore, if a hypersurface $M$ in $\mathbb E^{n+1}$ is biharmonic  or null 2-type, then the following system of differential equations is satisfied
\begin{subequations}
\begin{eqnarray}
\label{HSurfaceCondition}S(\nabla s_1)&=&-\frac {s_1}2\nabla s_1,\\
\label{Null2TypeCond} \Delta s_1&=&-s_1(s_1^2-2s_2-\lambda).
\end{eqnarray}
\end{subequations}
 Note that this is only a necessary condition. Hovewer, when this equation is analyzed, one can see that one of the principal directions of a biharmonic or null 2-type hypersurface is gradient of its mean curvature with corresponding principal curvature a constant multiple of the mean curvature. A hypersurface satisfying this interesting property is said to be an $\mbox{H}$-hypersurface, \cite{HsanisField} or biconservative hypersurface, \cite{Caddeo1,MontaldoArxiv}. Our opinion is that classifying $\mbox H$-hypersurfaces, or at least understanding their geometry, may play an important role on the theory of biharmonic hypersurfaces as well as null 2-type surfaces.

In this work, we study hypersurfaces with  3 distinct principal curvatures in the Euclidean space of arbitrary dimension. In Section 2, after we describe our notations, we give a summary of the basic facts and formulas that we will use. In Section 3, we obtain some geometrical properties of $\mbox{H}$-hypersurfaces. In Section 4, we give a classification  of $\mbox{H}$-hypersurfaces with  3 distinct principal curvatures. 

\section{Prelimineries}
Let $\mathbb E^m$ denote the Euclidean $m$-space with the canonical Euclidean metric tensor given by  
$$
\widetilde g=\langle\ ,\ \rangle=\sum\limits_{i=1}^m dx_i^2,
$$
where $(x_1, x_2, \hdots, x_m)$  is a rectangular coordinate system in $\mathbb E^m$. 

Consider an $n$-dimensional Riemannian submanifold  of the space $\mathbb E^m$. We denote Levi-Civita connections of $\mathbb E^m$ and $M$ by $\widetilde{\nabla}$ and $\nabla$, respectively. Then, the Gauss and Weingarten formulas are given, respectively, by
\begin{eqnarray}
\label{MEtomGauss} \widetilde\nabla_X Y&=& \nabla_X Y + h(X,Y),\\
\label{MEtomWeingarten} \widetilde\nabla_X \rho&=& -S_\rho(X)+\nabla^\perp_X \rho
\end{eqnarray}
for all tangent vectors fields $X,\ Y$ and normal vector fields $\rho$,  where $h$ and  $\nabla^\perp$ are the second fundamental form and the normal connection of $M$ in  $\mathbb E^m$,  respectively and     $S$ denotes the shape operator. Note that for each $\rho \in T^{\bot}_m M$, the shape operator $S_{\rho}$ along the normal direction $\rho$  is a symmetric  endomorphism of the tangent space  $T_m M$ at $m \in M$.   The shape operator and the second fundamental form are related by  
$$\left\langle h(X, Y), \rho\right\rangle = \left\langle S_{\rho}X, Y \right\rangle.$$
 
The Gauss and Codazzi equations are given, respectively, by
\begin{eqnarray}
\label{MinkGaussEquation}\label{GaussEq} \langle R(X,Y)Z,W\rangle&=&\langle h(Y,Z),h(X,W)\rangle-
\langle h(X,Z),h(Y,W)\rangle,\\
\label{MinkCodazzi} (\bar \nabla_X h )(Y,Z)&=&(\bar \nabla_Y h )(X,Z),
\end{eqnarray}
where  $R$ is the curvature tensor associated with connection $\nabla$ and  $\bar \nabla h$ is defined by
$$(\bar \nabla_X h)(Y,Z)=\nabla^\perp_X h(Y,Z)-h(\nabla_X Y,Z)-h(Y,\nabla_X Z).$$ 
The mean curvature vector $\zeta$ of $M$ is defined by 
$$\zeta=\frac 1n\mathrm{tr}\ h$$
and $\zeta$ is said to be parallel if $\nabla^\perp \zeta=0.$


\subsection{Hypersurfaces of Euclidean space}
Now, let $M$ be an oriented hypersurface in the Euclidean space $\mathbb E^{n+1}$,  $x$ its position vector and $S$ its shape operator along the unit normal vector field $N$ associated with the oriantiation of $M$. We consider a local orthonormal frame field $\{e_1,e_2,\hdots,e_n;N\}$ of consisting of  principal directions of $M$ with corresponding principal curvatures $k_1,\ k_2,\hdots,\ k_n$. We denote the dual basis of this frame field by $\{\theta_1,\theta_2,\hdots,\theta_n\}$. Then, the first structural equation of Cartan is
\begin{equation}
\label{CartansFirstStructural}d\theta_i=\sum\limits_{i=1}^n\theta_j\wedge\omega_{ij},\quad i=1,2,\hdots,n,
\end{equation}
where $\omega_{ij}$ denotes the connection forms corresponding to the chosen frame field, i.e., $\omega_{ij}(e_l)=\langle\nabla_{e_l}e_i,e_j\rangle$.

From the Codazzi equation \eqref{MinkCodazzi} we have
\begin{subequations}
\begin{eqnarray}
\label{Coda1}e_i(k_j)&=&\omega_{ij}(e_j)(k_i-k_j),\\
\label{Coda2}\omega_{ij}(e_l)(k_i-k_j)&=&\omega_{il}(e_j)(k_i-k_l)
\end{eqnarray}
\end{subequations}
for distinct $i,j,l=1,2,\hdots,n$. 

We put
$s_1=k_1+k_2+\cdots+k_n$
and, by abuse of terminology, we call this function as the (first) mean curvature of $M$.  Note that $M$ is said to be (1-) minimal if $s_1=0$.

\section{$\mbox{H}$-hypersurfaces}

In this section, we give some results on $\mbox{H}$-hypersurfaces of Euclidean spaces by extending the results obtained in \cite{HsanisField}.
\subsection{Connection forms of $\mbox{H}$-hypersurfaces}
Let $M$ be an $\mbox{H}$-hypersurface of the Euclidean space $\mathbb E^{n+1}$. Then, \eqref{HSurfaceCondition} is satisfied and $ s_1$ is not constant. We assume $\nabla s_1$ does not vanish at any point of $M$. From \eqref{HSurfaceCondition}, we have $\nabla s_1$ is a principal direction of $M$. We consider a frame field $\{e_1,e_2,\hdots,e_n\}$ of consisting of principal directions of $M$ with corresponding principal curvatures $k_1, k_2,\hdots, k_n$ such that $e_1=\nabla s_1/|\nabla s_1|$ and $k_1=-s_1/2$. Therefore, we have
\begin{equation}
\label{Section3Eq1} e_1(k_1)\neq0,\quad\quad e_x(k_1)=0,\quad x=2,3,\hdots,n
\end{equation}
and
\begin{equation}\label{Section3Eq0} 
3 k_1+k_2+k_3+\cdots+k_n=0.
\end{equation}

\begin{Remark}\label{Remmmark1}\cite{HsanisField}
If $k_1=k_x$ for some $2\leq x\leq,n$, then Codazzi equation \eqref{Coda1} for $i=1,\ j=x$ implies $e_1(k_1)=e_1(k_x)=\omega_{1x}(e_x)(k_1-k_x)=0$ which contradicts with \eqref{Section3Eq1}. Thus, the dimension of distribution $D_0$ given by 
$$D_0(m)=\{X\in T_m|SX=k_1X\}$$
is 1. The integral curves of $D_0$ are planar and geodesics of $M$. Furthermore, if $\alpha$ and $\beta$ are integral curves of $D_0$ passing through  $m$ and $m'$, respectively, then $\alpha$ and $\beta$ are congruent, \cite{HsanisField}.
\end{Remark}

By combaining \eqref{Section3Eq1} with Codazzi equation \eqref{Coda1} for $i=x,\ j=1$, we have
\begin{equation}
\label{Section3Eq2} \omega_{1x}(e_1)=0, \quad x=2,3,\hdots,n.
\end{equation}

On the other hand, we note that for a tangent vector field $X$ of $M$, $\langle X,e_1\rangle=0$ if and only if $Xk_1=0$. Therefore, because $[e_x,e_y](k_1)=0$, we have $\langle [e_x,e_y],e_1\rangle=0$ which implies
$$\omega_{1x}(e_y)=\omega_{1y}(e_x),\quad  \quad x,y=2,3,\hdots,n$$
with $x\neq y.$ From this equation and Codazzi equation \eqref{Coda2} for $i=1,\ j=x,\ l=z $ we get
\begin{equation}
\label{Section3Eq3} \omega_{1x}(e_y)=0, \quad x,y=2,3,\hdots,n\mbox{, } x\neq y.
\end{equation}
Therefore, \eqref{Coda2} for $i=x, j=y, l=1$ and \eqref{Coda2} for $i=x, j=1, l=y$ imply
\begin{subequations}\label{Section3Eq5ALL}
\begin{equation}
\label{Section3Eq5a} \omega_{xy}(e_1)=0, \quad x,y=2,3,\hdots,n\mbox{ if } k_x\neq k_y.
\end{equation}
In fact, we have
\begin{equation}
\label{Section3Eqb5} \omega_{xy}(e_z)=0, \quad x,y,z=1,2,3,\hdots,n\mbox{ if } x\neq z,\  k_x=k_z\neq k_y
\end{equation}
\end{subequations}
from the Codazzi equation \eqref{Coda2} for $i=x,\ j=y,\ l=z.$

Since  $\langle[e_1,e_x],e_1\rangle=0$ because of \eqref{Section3Eq2}, we have $[e_1,e_x](k_1)=0$ from which and \eqref{Section3Eq1} we obtain
\begin{equation}
\label{BihSurfq1Eq1} e_ie_1(k_1)=e_ie_1e_1(k_1)=0,\quad i=2,3,\hdots,n.
\end{equation}

\subsection{Some lemmas on $\mbox{H}$-hypersurfaces}
In this subsection, we obtain some lemmas that we will use on the rest of this paper. We also think that these lemmas can be useful for the future studies on biharmonic hypersurfaces. We think that these lemmas can be useful for the future studies on biharmonic hypersurfaces and null 2-type hypersurfaces.

First, we consider the distribution given by
\begin{equation}\label{KeyDistr}
D(m)=\{X\in T_mM|SX=k_2 X\}.
\end{equation}
\begin{Remark}
Obviously, the dimension of distribution $D$ given by \eqref{KeyDistr} is equal to multiplicity of $k_2$ as an eigenvalue of the shape operator $S$ of $M$.
\end{Remark}

We obtain the following lemma.
\begin{Lemma}\label{KeyLemma}
Let $M$ be an $\mbox{H}$-hypersurface in the Euclidean space $\mathbb E^{n+1}$ and  $k_2$ one of its principal curvatures. Then, the distribution $D$ in $M$ given by \eqref{KeyDistr} is involutive.
\end{Lemma}

\begin{proof}
If the dimension of $D$  is 1, then it is obviously involutive. We assume the dimension of $D$ is $p>1$. By renaming the indices if necessary, we assume 
\begin{equation}\label{REalKeyLemmaEq0} 
k_{2}=k_3=\hdots=k_{p+1}.
\end{equation}
We have $\langle \nabla_{e_A}e_B,e_i\rangle=\omega_{Bi}(e_A)=0$ for all $i=1,p+2,p+3,\hdots,n $ and $A,B=2,3,\hdots,p+1$ with $A\neq B$, because of \eqref{Section3Eqb5}. Thus, $(\nabla_{e_A}e_B)_m\in D(m)$ is satisfied from which we see that $X_m,Y_m\in D(m)$ implies $[X_m,Y_m]\in D(m)$. Hence, $D$ is involutive. 
\end{proof}

Now, we want to construct the integral submanifolds of distribution $D$ given by \eqref{KeyDistr}. We start by obtaining the following lemma.
\begin{Lemma}\label{REalKeyLemma}
Let $M$ be an $\mbox{H}$-hypersurface in the Euclidean space $\mathbb E^{n+1}$ and  $k_2$ one of its principal curvatures. Assume that the distribution $D$  given by \eqref{KeyDistr} has dimension greater than $1$. Then, any integral submanifold $H$ of $D$ has parallel mean curvature vector field on $\mathbb E^{n+1}$. Moreover, all of the shape operators of $H$ are proportional to identity operator. 
\end{Lemma}

\begin{proof}
Let the dimension of $D$  is $p>1$. Then the multiplicity of $k_2$ is $p$. Thus, by renaming indices if necessary, we assume \eqref{REalKeyLemmaEq0}. By using \eqref{Section3Eq5ALL} we obtain
\begin{subequations}\label{REalKeyLemmaEq1ALL} 
\begin{eqnarray}
\label{REalKeyLemmaEq1a} \widetilde\nabla_{e_A}e_A&=&-\omega_{1A}(e_A)e_1+\sum\limits_{C=2}^{p+1}\omega_{AC}(e_A)e_C+\sum\limits_{a=p+2}^{n}\omega_{Aa}(e_A)e_a+k_2N,\\
\label{REalKeyLemmaEq1b} \widetilde\nabla_{e_A}e_B&=&\sum\limits_{C=2}^{p+1}\omega_{BC}(e_A)e_C
\end{eqnarray}
\end{subequations}
for all $A,B=2,3,\hdots,p+1$ with $A\neq B$. Note that Codazzi equation \eqref{Coda1} for $i=1, j=A$ and $i=a, j=A$ give $\omega_{1A}(e_A)=\frac{e_1(k_A)}{k_1-k_A}$ and $\omega_{Aa}(e_A)=\frac{e_a(k_A)}{k_A-k_a}$, respectively. Thus, \eqref{REalKeyLemmaEq0} implies
\begin{subequations}\label{REalKeyLemmaEq2ALL} 
\begin{eqnarray}
\label{REalKeyLemmaEq2a} \xi&=&-\omega_{12}(e_2)=-\omega_{13}(e_3)=\hdots=-\omega_{1(p+1)}(e_{p+1}),\\
\label{REalKeyLemmaEq2b} \eta_a&=&\omega_{a2}(e_2)=\omega_{a3}(e_3)=\hdots=\omega_{a(p+1)}(e_{p+1})
\end{eqnarray}
\end{subequations}
for some functions $\xi$ and $\eta_a$ for $a=p+2,p+3,\hdots,n$.

Now, let $H$ be an integral submanifold of $D$ and consider the local orthonormal frame field 
$$\{f_1,f_2,\hdots,f_p; f_{p+1},f_{p+2},\hdots f_{n+1}\}$$
on $H$ given by 
\begin{equation}\label{FrameFieldonH}
f_{A-1}=\left.e_A\right|_{H},\ f_{p+1}=\left.e_1\right|_{H},f_{a}=\left.e_a\right|_{H}, f_{n+1}=\left.N\right|_{H}.
\end{equation}

From \eqref{REalKeyLemmaEq1ALL} and \eqref{REalKeyLemmaEq2ALL}, we have
\begin{subequations}\label{REalKeyLemmaEq3ALL} 
\begin{eqnarray}
\widetilde\nabla_{f_i}f_i&=&\hat\nabla_{f_i}f_i+\hat\xi f_{p+1}+\sum\limits_{a=p+2}^{n}\hat\eta_a f_a+\hat k_2f_{n+1}\\
\widetilde\nabla_{f_i}f_j&=&\hat\nabla_{f_i}f_j,\quad i,j=1,2,\hdots,p,\ i\neq j,
\end{eqnarray}
\end{subequations}
where $\hat\nabla$ denotes the Levi-Civita connection of $H$, $\hat\xi,\hat\eta_a$ and $\hat k_2$ are restrictions of $\xi,\eta_a$ and $k_2$ to $H$, respectively.

Therefore, we have 
\begin{equation}\label{REalKeyLemmaEqAAA} 
\hat S_{p+1}=\hat\xi I,\  \hat S_{a}=\hat\eta_a I,\ \hat S_{n+1}=\hat k_2 I 
 \end{equation}
or, equivalently,  
\begin{equation}\label{REalKeyLemmaEq4} 
\zeta=\hat h(f_1,f_1)=\hat h(f_1,f_1)=\hdots=\hat h(f_p,f_p)=\hat\xi f_{p+1}+\sum\limits_{a=p+2}^{n}\hat\eta_a f_a+\hat k_2f_{n+1},
\end{equation}
where $\hat h$ stands for the second fundemental form of $H$ in $\mathbb E^{n+1}$, $\hat S_\alpha$ denotes the shape operator  of $H$ in $\mathbb E^{n+1}$ along the normal vector field $f_\alpha$ and $\zeta$ is the mean curvature vector of $H$  in $\mathbb E^{n+1}$.

Furthermore, Codazzi equation \eqref{MinkCodazzi} for $X=Z=f_i$ and $Y=f_j$ for $i\neq j$ gives
$$\hat\nabla^\perp_{f_i}\hat h(f_i,f_j)-\hat h(\hat\nabla_{f_i}f_i,f_j)-\hat h(f_i,\hat\nabla_{f_i}f_j)=\hat\nabla^\perp_{f_j}h(f_i,f_i)-2\hat h(\hat\nabla_{f_j}f_i,f_i),$$
where $\hat\nabla^\perp$ is the normal connection of $H$ in $\mathbb E^{n+1}$. By using \eqref{REalKeyLemmaEq3ALL} in this equation and considering \eqref{REalKeyLemmaEq4}, we get $\hat\nabla^\perp_{f_j}\zeta=0$ for all $j=1,2,\hdots,p$. Hence, the mean curvature vector $\zeta$ of $H$ is parallel.
\end{proof}

\begin{Remark}\label{RemarkafterKEYLemma}
Let $M$ be an $\mbox{H}$-hypersurface in the Euclidean space $\mathbb E^{n+1}$ and  $k_2$ one of its principal curvatures. Assume that the distribution $D$  given by \eqref{KeyDistr} has dimension greater than $1$ and $H$ is an (connected) integral submanifold of $D$. Then, from the Gauss equation \eqref{MinkGaussEquation} for $X=e_A,Y=e_B, Z=e_x, W=e_A$ we obtain $e_B(\omega_{xA}(e_A))=0$, for all $A,B=2,3,\hdots,p+1$ and $x=1,p+2,p+3,\hdots,n$. Therefore, the functions $k_2$, $\xi=\omega_{1A}(e_A)$ and $\eta_a=\omega_{Aa}(e_A)$ are constant on $H$.
\end{Remark}

By the following proposition, we obtain the integral submanfiolds of distribution $D$ given by \eqref{KeyDistr}.
\begin{Prop}\label{KeyPropIntSubman}
Let $M$ be an $\mbox{H}$-hypersurface in the Euclidean space $\mathbb E^{n+1}$ and  $k_2$ one of its principal curvatures. Assume that the distribution $D$  given by \eqref{KeyDistr} has dimension $p>1$ and $H$ is an (connected) integral submanifold of $D$ passing through $m\in M$. If  $k_2(m)=0$ and $(\nabla k_2)_m=0$ then $H$ is a $p$-plane of $\mathbb E^{n+1}$. Otherwise, $H$ lies on a $(p+1)$-plane of $\mathbb E^{n+1}$ and it is congruent to a hypersphere of  $\mathbb E^{p+1}$.
\end{Prop}
\begin{proof}
First, suppose that $k_2$ and $\nabla k_2$ vanish at $m$. Then, we have $\hat\eta_a(m)=\hat\xi(m)=0$ for $a=p+2,p+3,\hdots,n$, where $\hat\eta_a$ and $\hat\xi$ are functions defined in the proof of Lemma \ref{REalKeyLemma}. Remark \ref{RemarkafterKEYLemma} implies that  $\hat k_2\equiv0$, $\hat \xi\equiv0$ and $\hat \eta\equiv0$ on $H$. Thus, \eqref{REalKeyLemmaEqAAA} implies $\hat h=0$, i.e., $M$ is a totally geodesic $p$-dimensional submanifold of $\mathbb E^{n+1}$. Hence, $M$ is a $p$-plane.

Next, assume $k_2(m)\neq0$. Define $n-p$ normal vector fields $\zeta_1,\zeta_{p+2},\hdots,\zeta_{n}$ by 
$\zeta_1= \hat k_2 f_p-\hat\xi f_{n+1}$ and $\zeta_a=\hat k_2f_a-\hat\eta_af_{n+1}$. Clearly, $\zeta_1,\zeta_{p+2},\hdots,\zeta_{n}$ are linearly independent constant vector fields normal to $H$. Thus, $H$ lies in a $(p+1)$-plane $\Pi\cong\mathbb E^{p+1}$ of $\mathbb E^{n+1}$. As its mean curvature vector is parallel, and shape operator is proportional to identity operator $I$, it is a hypersphere of $\Pi$.

If $(\nabla k_2)_m\neq0$, then we have $\hat\xi(m)\neq0$  or $\hat\eta_a(m)\neq0$ for some $a$ because of Codazzi equation \eqref{Coda1}. The same proof can be done for both cases.
\end{proof}


\section{$\mbox{H}$-hypersurfaces with 3 distinct principal curvatures}
In this subsection, we mainly focus on hypersurfaces with  $3$ distinct curvatures.

Let $M$ be an $\mbox{H}$-hypersurfaces in $\mathbb E^{n+1}$ and $x$ its position vector.  Since the study for hypersurfaces with 2 distinct principal curvatures are completed in \cite{HsanisField}, we assume that the shape operator $S$ of $M$ is given by 
\begin{equation}\label{HSurfaceATMOST3Eq1}
S=\mathrm{diag}(k_1,\underbrace{k_2,k_2,\hdots,k_2}_{p\mbox{ times}},\underbrace{k_{p+2},k_{p+2},\hdots,k_{p+2}}_{q\mbox{ times}}),\quad k_2\neq k_{p+2}
\end{equation}
 corresponding to the local orthonormal frame field $\{e_1,e_2,\hdots,e_n\}$ of consisting of  principal directions of $M$ and the functions $k_1-k_2$, $k_1-k_{p+2}$ and $k_2-k_{p+2}$ do not vanish on $M$, where $q=n-p-1$.  

First, we consider the distribution $D^\perp$ given by 
\begin{equation}\label{HSurfaceATMOST3Eq4AA}
D^\perp(m)=\{X_m\in T_mM| \langle X_m,Y\rangle=0, \mbox{ for all }Y\in D(m)\}.
\end{equation}

\begin{Lemma}
Let $M$ be an $\mbox{H}$-hypersurface in the Euclidean space $\mathbb E^{n+1}$ with the shape operator given by \eqref{HSurfaceATMOST3Eq1}. Then, the distribution $D^\perp$ in $M$ given by \eqref{HSurfaceATMOST3Eq4AA} is involutive.
\end{Lemma}
\begin{proof}
From the definition, we have $D^\perp(m)=\mathrm{span}\{(e_1)_m,(e_{p+2})_m,(e_{p+3})_m,\hdots,(e_{n})_m\}$. Moreover, from \eqref{Section3Eq5ALL} and \eqref{HSurfaceATMOST3Eq1} we have $\nabla_{e_a}e_b,\nabla_{e_a}e_1, \nabla_{e_1}e_a\in D^\perp$ for all $a,b=p+2,p+3,\hdots,n$. Thus, for all $X,Y\in D^\perp$, we have $[X,Y]\in D^\perp$. Hence $D^\perp$ is involutive.
\end{proof}

\begin{Remark}\label{RemarkCoords}
By combaining \eqref{Section3Eq2} and \eqref{Section3Eq3} with Cartan's first structural equation \eqref{CartansFirstStructural}, we obtained $d\theta_1=0$, i.e., $\theta_1$ is closed. Thus, Poincar\`e lemma implies that $d\theta_1$ is exact, i.e., there exists a function $s$ such that $\theta_1=ds$. Moreover, since the distributions $D$ and $D^\perp$ given by \eqref{KeyDistr} and \eqref{HSurfaceATMOST3Eq4AA} are involutive, there exists a local coordinate system $t_1,t_2,\hdots,t_n$  on a neighborhood of $m\in M$ such that $t_2,t_3,\hdots,t_{p+1}$ span $D$ and $t_1=s,t_{p+2},t_{p+3},\hdots,t_{n}$ span $D^\perp$ because of the local Frobenius theorem. Thus, by redefining the vector fields $e_i,\ i=2,3,\hdots,n$ properly, we can assume
\begin{equation}
\label{Section3Eq4} e_1=\partial_s,\quad e_i=F_i\partial_{t_i},\ i=2,3,\hdots,n
\end{equation}
 for some smooth non-vanishing functions $F_i=F_i(s,t_2,t_3,\hdots,t_n)$.
\end{Remark}

Since the study on $\mathbb E^4$ is completed in \cite{HsanisField}, we assume $n>3$. Thus, we may assume $p>1$ or $q>1$. Without loss of generality, we assume $p>1$.  Thus, we have $k_2=k_3$  from which and Codazzi equation \eqref{Coda1} we obtain 
\begin{eqnarray}
\label{HSurfaceATMOST3Eq2a}e_A(k_2)&=&0, \quad A=2,3,\hdots,p+1.
\end{eqnarray}
From \eqref{Section3Eq1},  \eqref{Section3Eq0}, \eqref{HSurfaceATMOST3Eq1} and \eqref{HSurfaceATMOST3Eq2a} we also get
\begin{eqnarray}
\label{HSurfaceATMOST3Eq3a} e_A(k_{p+2})&=&0, \quad a=p+2,p+3,\hdots,n
\end{eqnarray}
from which and Codazzi equation \eqref{Coda1} for $i=A,j=a$ we obtain 
\begin{equation}
\label{HSurfaceATMOST3Eq4a}\omega_{Aa}(e_a)=0.
\end{equation}

\subsection{Case $p>1$ and $q>1$}
First, we want to deal the case that $q>1$ and obtain all $\mbox H$-hypersurfaces in Euclidean space $\mathbb E^{n+1},\ n>3$ with shape operator \eqref{HSurfaceATMOST3Eq1}. In this case, we have  $k_{p+2}=k_{p+3}$ from which and Codazzi equation \eqref{Coda1} we obtain 
\begin{eqnarray}
\label{HSurfaceATMOST3Eq2b}e_a(k_{p+2})&=&0, \quad a=p+2,p+3,\hdots,n
\end{eqnarray}
 and \eqref{Section3Eq1},  \eqref{Section3Eq0}, \eqref{HSurfaceATMOST3Eq1} and \eqref{HSurfaceATMOST3Eq2b} imply
\begin{eqnarray}
\label{HSurfaceATMOST3Eq3b} e_a(k_{2})&=&0, \quad a=p+2,p+3,\hdots,n.
\end{eqnarray}
Therefore, from Codazzi equation \eqref{Coda1} $i=a,j=A$ we obtain 
\begin{equation}
\label{HSurfaceATMOST3Eq4b}\omega_{Aa}(e_A)=0, \quad A=2,3,\hdots, p+1,\ a=p+2,p+3,\hdots,n.
\end{equation}

By combaining  \eqref{Section3Eq2}, \eqref{Section3Eq5ALL}, \eqref{Section3Eq3}, \eqref{HSurfaceATMOST3Eq4a} and \eqref{HSurfaceATMOST3Eq4b}, we get 
\begin{subequations}\label{ANEquatioNNEq2All} 
\begin{eqnarray}
\label{ANEquatioNNEq2a} \widetilde\nabla_{e_1}e_A&=&\sum\limits_{C=2}^{p+1}\omega_{AC}(e_1)e_C, \\
\label{ANEquatioNNEq2b} \widetilde\nabla_{e_A}e_1&=&\omega_{12}(e_2)e_A,\\
\label{ANEquatioNNEq2f} \widetilde\nabla_{e_a}e_1&=&\omega_{1{(p+2)}}(e_{p+2})e_a,
\end{eqnarray}
\end{subequations}
for all  $a=p+2,p+3,\hdots,n$ and $A,B=2,3,\hdots,p+1$ with $A\neq B$, where $N$ is the unit normal field of $M$.

On the other hand, since $[e_1,e_A](k_2)=e_Ae_1(k_2)$, from \eqref{ANEquatioNNEq2a} and \eqref{ANEquatioNNEq2b} we have
$$e_Ae_1(k_2)=\left(\sum\limits_{C=2}^{p+1}\omega_{AC}(e_1)e_C-\omega_{1A}(e_A)e_A\right)(k_2).$$
The right hand-side of this equation is zero because of \eqref{HSurfaceATMOST3Eq2a}. Thus, we obtain
\begin{equation}
\label{HSurfaceCasep2PosVecEqek2}e_Ae_1(k_2)=0.
\end{equation}
Furthermore, from \eqref{Coda1} for $i=1, j=2$ we have 
$$e_A(\omega_{12}(e_2))=e_A\left(\frac{e_1(k_2)}{k_1-k_2}\right).$$
By combaining this equation with \eqref{Section3Eq1}, \eqref{HSurfaceATMOST3Eq2a} and \eqref{HSurfaceCasep2PosVecEqek2} we get 
\begin{subequations}\label{HSurfaceCasep2PosVecEq3} 
\begin{equation}\label{HSurfaceCasep2PosVecEq3a} 
e_A(\omega_{12}(e_2))=0. 
\end{equation}
By a similar way, we obtain
\begin{eqnarray}
e_a(\omega_{12}(e_2))&=&0,\\
\label{HSurfaceCasep2PosVecEq3c}e_A(\omega_{1(p+2)}(e_{p+2}))=e_a(\omega_{1(p+2)}(e_{p+2}))&=&0.
\end{eqnarray}
\end{subequations}
By combaining the equations in \eqref{HSurfaceCasep2PosVecEq3}, we get
\begin{equation}\label{HSurfaceCasep2PosVecEq3v2} 
\omega_{12}(e_2)=\xi(s),\quad\quad  \omega_{1(p+2)}(e_{p+2})=\eta(s)
\end{equation}
for some functions $\xi,\eta$, where $s$ is the local coordinate given in Remark \ref{RemarkCoords}.

Next, we  want to obtain the position vector of an $\mbox{H}$-surface.
\begin{theorem}\label{HSurfClassTheo}
Let $M$ be a hypersurface in $\mathbb E^{n+1}$ with the shape operator given by \eqref{HSurfaceATMOST3Eq1}, $k_2\neq k_{p+2}$ and $p>1, q>1$. Then, $M$ is an $\mbox{H}$-hypersurface if and only if it is congruent to one of the following hypersurfaces.
\begin{enumerate}
\item[(i)] A generalized rotational hypersurface given by
\begin{align}\label{HSurfaceCasep2PosVecCase1}
\begin{split}
x(s, t_2,\hdots,t_n)=&\big(\psi(s)\cos t_2,\psi(s)\sin t_2\cos t_3,\hdots,\psi(s)\sin t_2\hdots\sin t_{p}\cos t_{p+1},\\
&\psi(s)\sin t_{2}\hdots\sin t_{p}\sin t_{p+1}, \phi(s)\cos t_{p+2},\phi(s)\sin t_{p+2}\cos t_{p+3},\hdots,\\
&\phi(s)\sin t_{p+2}\hdots\sin t_{n-1}\cos t_n,\phi(s)\sin t_{p+2}\hdots\sin t_{n-1}\sin t_n\big)
\end{split}
\end{align}
with the profile curve $(\psi,\phi)$ satisfying $\psi'^2+\phi'^2=1$ and
\begin{equation}\label{HSurfaceCasep2PosVecCase1Cond} 
\phi'\psi''-\phi''\psi'=\frac{1}{3}\left(p\frac{\phi'}{\psi}-q\frac{\psi'}{\phi}\right).
\end{equation}

\item[(ii)] A generalized cylinder over a rotational hypersurface given by 
\begin{align}\label{HSurfaceCasep2PosVecCase2}
\begin{split}
x(s, t_1,\hdots,t_n)=&\big(\psi(s)\cos t_2,\psi(s)\sin t_2\cos t_3,\hdots,\psi(s)\sin t_2\hdots\sin t_{p}\cos t_{p+1},\\
&\psi(s)\sin t_{2}\hdots\sin t_{p}\sin t_{p+1}, \phi(s),t_{p+2},t_{p+3},\hdots,t_{n}\big)
\end{split}
\end{align}
 with the profile curve $(\psi,\phi)$ satisfying $\psi'^2+\phi'^2=1$ and
\begin{equation}\label{HSurfaceCasep2PosVecCase2Cond} 
\phi'\psi''-\phi''\psi'=\frac{p}{3}\frac{\phi'}{\psi}.
\end{equation}
\end{enumerate}
\end{theorem}

\begin{proof}
We assume that $M$ is an $\mbox{H}$-hypersurface. Then,  \eqref{Section3Eq0} is satisfied. Let $s,t_2,t_3,\hdots,t_n$ be the local coordinate system given in Remark \ref{RemarkCoords}. From \eqref{ANEquatioNNEq2b} and \eqref{ANEquatioNNEq2f} we have
\begin{subequations}\label{HSurfaceCasep2PosVecEq1All} 
\begin{eqnarray}
\label{HSurfaceCasep2PosVecEq1a} x_{st_A}&=&\omega_{12}(e_2)x_{t_A},\quad A=2,3,\hdots,p+1 \\
\label{HSurfaceCasep2PosVecEq1b} x_{st_a}&=&\omega_{1(p+2)}(e_{p+2})x_{t_a},\quad a=p+2,p+3,\hdots,n.
\end{eqnarray}
\end{subequations}

By taking into account the \eqref{HSurfaceCasep2PosVecEq3v2}, we integrate \eqref{HSurfaceCasep2PosVecEq1All} to obtain
$$x_s=\xi(s)x+\tilde\Theta_2(s,t_{p+2},t_{p+3},\hdots,t_n)=\eta(s)x+\tilde\Theta_1(s,t_2,t_3,\hdots,t_{p+1})$$
for some vector valued functions $\tilde\Theta_1,\tilde\Theta_2$. Therefore, we have
\begin{equation}\label{HSurfaceCasep2PosVecEq2} 
x(s,t_2,t_3,\hdots,t_n)=\hat\Theta_1(s,t_2,t_3,\hdots,t_{p+1})+\hat\Theta_2(s,t_{p+2},t_{p+3},\hdots,t_n)
\end{equation}
for some  vector valued functions $\hat\Theta_1$ and $\hat\Theta_2$.
 
Next, we put \eqref{HSurfaceCasep2PosVecEq2} in \eqref{HSurfaceCasep2PosVecEq1All} and get
$$\hat\Theta_{1,st_A}= \xi(s)\hat\Theta_{1,t_A},\quad \hat\Theta_{2,st_a}=\eta(s)\hat\Theta_{2,t_a}.$$
By integrating these equations we obtain 
$$x(s,t_2,t_3,\hdots,t_n)=\psi(s)\Theta_1(t_2,t_3,\hdots,t_{p+1})+\phi(s)\Theta_2(t_{p+2},t_{p+3},\hdots,t_n)+\varphi(s)$$
for some functions functions $\phi,\psi$ and vector valued functions $\Theta_1,\Theta_2,\varphi$.
By taking into account Remark \ref{Remmmark1}, we see that we see that $\varphi$ is a constant vector. Thus, we assume $\varphi=0.$ Thus, we have 
\begin{equation}\label{HSurfaceCasep2PosVecPre1}
x(s,t_2,t_3,\hdots,t_n)=\psi(s)\Theta_1(t_2,t_3,\hdots,t_{p+1})+\phi(s)\Theta_2(t_{p+2},t_{p+3},\hdots,t_n).
\end{equation}
Because  of \eqref{Section3Eq4}, we have 
\begin{subequations}\label{HSurfaceCasep2PosVecPre1xxEqALL}
\begin{eqnarray}
\label{HSurfaceCasep2PosVecPre1xxEqALL1a}\left\langle \Theta_{1,t_A}, \Theta_{2,t_a}\right\rangle&=&0,\\
\label{HSurfaceCasep2PosVecPre1xxEqALL1c}\langle x_s,x_s\rangle&=&1
\end{eqnarray}
\end{subequations}
for all $A=2,3,\hdots,p+1$ and $a=p+2,p+3,\hdots,n$.

Since $k_2\neq k_{p+2}$, without loss of generality, we may assume $k_2\neq 0.$ Now, we consider the slice $H$ of $M$ given by
$$y_1(t_2,t_3,\hdots,t_{p+1})=x(\bar s,t_2,t_3,\hdots,t_{p+1},\bar t_{p+2},\bar t_{p+3},\bar t_{n})$$
 passing through the point  $m=x(\bar s,\bar t_2,\bar t_3,\hdots,\bar t_n)\in M$. From \eqref{HSurfaceCasep2PosVecPre1} we have 
\begin{eqnarray}\label{HSurfaceCasep2PosVecPre2}
y_1(t_2,t_3,\hdots,t_{p+1})=c_0\Theta_1(t_2,t_3,\hdots,t_{p+1})+v_0,
\end{eqnarray}
where $c_0=\phi(\bar s)$ is a constant and $v_0=\psi(\bar s)\Theta_2(\bar t_{p+2},\bar t_{p+3},\hdots,\bar t_n)$ is a constant vector.

Since $H$ is an integral curve of the distribution $D$ given by \eqref{KeyDistr}, and $k_2\neq0$, it is congruent to hypersphere of $\mathbb E^{p+1}$ because of Proposition \ref{KeyPropIntSubman}. Thus, by choosing suitable coordinates and redefining $\psi$, we may assume
\begin{align}\label{HSurfaceCasep2PosVecPre3}
\begin{split}
\Theta_1(t_2,\hdots,t_{p+2})=&\big(\cos t_2,\psi\sin t_2\cos t_3,\hdots,\psi\sin t_2\hdots\sin t_{p}\cos t_{p+1},\\
&\psi\sin t_{2}\hdots\sin t_{p}\sin t_{p+1}, 0,0,\hdots,0\big).
\end{split}
\end{align}

Now, consider the submanifold $H'$  given by 
$$y_2(t_{p+2},t_{p+3},\hdots,t_{n})=x(\bar s,\bar t_2,\bar t_3,\hdots,\bar t_{p+1},t_{p+2},t_{p+3}, t_{n})$$
which is an integral submanifold of distribution $D$ given by $D'(m')=\{X\in T_{m'}M|SX=k_{p+2}X\}$ passing through the point $m$. Now, we have two cases: $k_{p+2}=0$ and $k_{p+2}\neq0$.

\textbf{Case 1.}  $k_{p+2}=0$. In this case, $H'$ is a $q$-plane because of Proposition \ref{KeyPropIntSubman}. Thus, $\Theta_2$ is the position vector of a $p$-plane. Because of \eqref{HSurfaceCasep2PosVecPre1xxEqALL1a} without loss of generality, we may assume
\begin{align}\nonumber
\begin{split}
\Theta_2(t_{p+2},t_{p+3},\hdots,t_{n})=&\big(0,0,\hdots, 1,t_{p+2},t_{p+3},\hdots,t_{n}\big).
\end{split}
\end{align}
Therefore, we obtain \eqref{HSurfaceCasep2PosVecCase2}. Because of \eqref{HSurfaceCasep2PosVecPre1xxEqALL1c}, we have $\psi'^2+\phi'^2=1$. 

Moreover, the shape operator of this hypersurface is 
\begin{equation}\label{HSurfaceCasep2PosVecCase2ShapeOp}
S=\mathrm{diag}(k_1,\underbrace{\frac{\phi'}{\psi},\frac{\phi'}{\psi},\hdots,\frac{\phi'}{\psi}}_{p\mbox{ times}},\underbrace{0,0,\hdots,0}_{q\mbox{ times}}).
\end{equation}
From \eqref{Section3Eq0} and \eqref{HSurfaceCasep2PosVecCase2ShapeOp} we get \eqref{HSurfaceCasep2PosVecCase2Cond}. Hence, we have case (ii) of theorem.

\textbf{Case 2.}  $k_{p+2}\neq0$. In this case, $H'$ is congruent to a hypersphere of $\mathbb E^{q+1}$ because of Proposition \ref{KeyPropIntSubman}. Because of \eqref{HSurfaceCasep2PosVecPre1xxEqALL1a}, without loss of generality, we choose
\begin{align}\nonumber
\begin{split}
\Theta_1(t_2,\hdots,t_{p+2})=&\big(\underbrace{0,0,\hdots,0}_{p+1\mbox{ times}},\cos t_{p+2},\psi\sin t_{p+2}\cos t_{p+3},\hdots,\psi\sin t_{p+2}\hdots\sin t_{p}\cos t_{p+1},\\
&\psi\sin t_{2}\hdots\sin t_{p}\sin t_{p+1}, \big).
\end{split}
\end{align}
Therefore, we obtain \eqref{HSurfaceCasep2PosVecCase1}. Because of \eqref{HSurfaceCasep2PosVecPre1xxEqALL1c}, we have $\psi'^2+\phi'^2=1$. 

Moreover, the shape operator of this hypersurface is 
\begin{equation}\label{HSurfaceCasep2PosVecCase2ShapeOp2}
S=\mathrm{diag}(k_1,\underbrace{\frac{\phi'}{\psi},\frac{\phi'}{\psi},\hdots,\frac{\phi'}{\psi}}_{p\mbox{ times}},\underbrace{-\frac{\psi'}{\phi},-\frac{\psi'}{\phi},\hdots,-\frac{\psi'}{\phi}}_{q\mbox{ times}}).
\end{equation}
From \eqref{Section3Eq0} and \eqref{HSurfaceCasep2PosVecCase2ShapeOp2} we get \eqref{HSurfaceCasep2PosVecCase1Cond}. Hence, we have case (i) of theorem.
\end{proof}

\begin{Remark}
 In \cite{MontaldoArxiv}, it was proved that none of these type of hypersurfaces are biharmonic. Recently, in \cite{FuYu2014}, Yu Fu remarked that he extended this result by proving that there is no non-minimal biharmonic hypersurface in $\mathbb E^{n+1}$ with 3 distinct principal curvature. However, classifying null 2-type hypersurfaces with  3 distinct principal curvature is an open problem.
\end{Remark}
\subsection{Case $p>1$ and $q=1$}
In the remaining part, we will consider the case $p>1$ and $q=1$ to obtain a necessary condition for null 2-type hypersurfaces with 3 principal curvatures. The shape operator of $M$ is 444
\begin{equation}\label{Null2TSurfShapeOp}
S=\mathrm{diag}(k_1,\underbrace{k_2,k_2,\hdots,k_2}_{p\mbox{ times}},k_n).
\end{equation}
Since $p>1$, the equations \eqref{HSurfaceATMOST3Eq2a}-\eqref{HSurfaceATMOST3Eq4a} are still satisfied. Moreover,  the distribution $D$ given in \eqref{KeyDistr} is involutive and its integral submanifold are congruent to hyperspheres or hyperplanes of $\mathbb E^{n-1}$ because of Lemma \ref{KeyLemma} and Lemma \ref{REalKeyLemma}. From \cite[Lemma 2.2]{HsanisField}, we also know that the integral curves of $e_1=\partial_s$ are some planar curves and congruent to each other. Therefore, we first want to  focus on the remaining part, integral curves of $e_n$.

Let $M$ be a hypersurface with the shape operator given in \eqref{Null2TSurfShapeOp}. We also suppose that the functions $k_1-k_2,$ $k_1-k_n$ and $k_2-k_n$ do not vanish on $M$. Now, assume that $M$ is a null 2-type hypersurface. Then, $M$ is an $\mbox{H}$-surface satisfying \eqref{Null2TypeCond}. Moreover, from \eqref{Section3Eq0} and \eqref{Null2TSurfShapeOp} we have
\begin{equation}
\label{Null2TypeSurfq1Eq2b} 3k_1+(n-2)k_2+k_n=0
\end{equation}
because $M$ is an $\mbox{H}$-surface.

By combaining \eqref{HSurfaceATMOST3Eq2a} and \eqref{HSurfaceATMOST3Eq3a} with Codazzi equation \eqref{Coda1} we have $\omega_{An}(e_n)=0$.
 Therefore, we have 
\begin{subequations}\label{Null2TypeSurfq1Eq1}
\begin{eqnarray}
\widetilde\nabla_{e_n}e_1=\omega_{1n}(e_n)e_n,&\quad& \widetilde\nabla_{e_A}e_1=\omega_{1A}(e_A)e_A\\
\label{Null2TypeSurfq1Eq1b}\widetilde\nabla_{e_A}e_n=-\omega_{An}(e_A)e_A,&\quad& \widetilde\nabla_{e_n}e_A=\omega_{AB}(e_n)e_B
\end{eqnarray}
\end{subequations}
Now, we  want to show $e_n(k_2)=0$ by using a  method similar with \cite{FuYu2014}. 

Since $e_A(k_2)=0$ and $e_A(k_n)=0$, we have $[e_A,e_1](k_2)=e_Ae_1(k_2)$ and $[e_A,e_1](k_n)=e_Ae_1(k_n)$. By computing the left-hand side of each of these equations using  \eqref{Null2TypeSurfq1Eq1b}, we get
\begin{equation}
\label{Null2TypeSurfq1Eq2} e_Ae_1(k_2)=e_Ae_1(k_n)=0,\quad A=2,3,\hdots,n-1.
\end{equation}

Furthermore, from the Gauss equation \eqref{MinkGaussEquation} for $X=e_A,$ $Y=e_n,$ $Z=e_1,$ and $W=e_A$ we obtained
\begin{equation}
\label{Null2TypeSurfq1Eq3} e_n(\omega_{12}(e_2))=\frac{e_n(k_2)}{k_2-k_n}(\omega_{12}(e_2)-\omega_{1n}(e_n)).
\end{equation}
By a direct calculation using Codazzi equation \eqref{Coda1}, \eqref{Section3Eq1}, \eqref{Null2TypeSurfq1Eq2b} and \eqref{Null2TypeSurfq1Eq3}, we also obtain
\begin{equation}
\label{Null2TypeSurfq1Eq4} e_n(\omega_{1n}(e_n))=(n-2)\frac{(2k_2-k_1-k_n)e_n(k_2)}{(k_1-k_n)(k_2-k_n)}(\omega_{12}(e_2)-\omega_{1n}(e_n)).
\end{equation}

On the other hand, from \eqref{Null2TSurfShapeOp} and  \eqref{Null2TypeCond} we have
\begin{equation}
\label{BiharmonicEq1v00}e_1e_1(k_1)+(n-2)\omega_{12}(e_2)e_1(k_1)+\omega_{1n}(e_n)e_1(k_1)=k_1(k_1^2+(n-2)k_2^2+k_n^2-\lambda).
\end{equation}
By applying $e_n$ to both hand side of this equation and using \eqref{BihSurfq1Eq1}, \eqref{Null2TypeSurfq1Eq3} and \eqref{Null2TypeSurfq1Eq4} we obtain
\begin{equation}
\label{BiharmonicEq1v00ara}
e_n(k_n)\Big(e_1(k_1)(\omega_{12}(e_2)-\omega_{1n}(e_n))-k_1(k_2-k_n)(k_1-k_n))\Big)=0.
\end{equation}
From the assumptions, we have the functions $\omega_{12}(e_2)-\omega_{1n}(e_n)$ and $k_1$ do not vanish. Thus,  if $e_n(k_n)\neq0$, then we have
$$\frac{e_1(k_1)}{k_1}=\frac{(k_2-k_n)(k_1-k_n)}{\omega_{12}(e_2)-\omega_{1n}(e_n)}$$ 
because of \eqref{BiharmonicEq1v00ara}. By applying $e_n$ to this equation we obtain
$$e_n\left(\frac{(k_2-k_n)(k_1-k_n)}{\omega_{12}(e_2)-\omega_{1n}(e_n)}\right)=0.$$
Next, we compute the left-hand side of this equation by using \eqref{Section3Eq1}, \eqref{Null2TypeSurfq1Eq2b}, \eqref{Null2TypeSurfq1Eq3} and \eqref{Null2TypeSurfq1Eq4} to get $k_n=a_0k_2$ for a constant $a_0$. However, this equation,  \eqref{Section3Eq1} and  \eqref{Null2TypeSurfq1Eq2b} give us $e_n(k_2)=0$ which is a contradiction. Therefore, we have 
\begin{equation}
\label{Null2TypeSurfq1EqRESULT1a} e_n(k_2)=0
\end{equation}
and \eqref{Section3Eq1}, \eqref{Null2TypeSurfq1Eq2b} imply
\begin{equation}
\label{Null2TypeSurfq1EqRESULT1b} e_n(k_n)=0.
\end{equation}
Moreover, from Codazzi equation \eqref{Coda1} and \eqref{Null2TypeSurfq1EqRESULT1a} we have 
\begin{equation}
\label{Null2TypeSurfq1EqRESULT1c} \omega_{An}(e_A)=0.
\end{equation}
On the other hand, from \eqref{Null2TypeSurfq1Eq3}, \eqref{Null2TypeSurfq1Eq4} and \eqref{Null2TypeSurfq1EqRESULT1a} we get
\begin{subequations}\label{Null2TypeSurfq1EqRESULT1ALL}
\begin{equation}
\label{Null2TypeSurfq1EqRESULT1d} e_n(\omega_{1n}(e_n))= e_n(\omega_{1A}(e_A))=0
\end{equation}
and by taking into account \eqref{Null2TypeSurfq1Eq1} and using Gauss equation \eqref{MinkGaussEquation} for $X=e_A,$ $Y=e_n,$ $Z=e_1,$ $W=e_n$ we obtain 
\begin{equation}
\label{Null2TypeSurfq1EqRESULT1e} e_A(\omega_{1n}(e_n))=0.
\end{equation}
From Codazzi equation \eqref{Coda1} for $i=1,\ j=A$ we have $\omega_{1A}(e_A)=e_1(k_A)/(k_1-k_A)$. Thus, we have
\begin{equation}
\label{Null2TypeSurfq1EqRESULT1f} e_A(\omega_{1A}(e_A))=0.
\end{equation}
\end{subequations}

Next, we want to give a geometric interpretation of these results.
\begin{Prop}\label{BihSurfCircl}
Let $M$ be a null 2-type hypersurface in $\mathbb E^{n+1}$ with shape operator given by \eqref{Null2TSurfShapeOp}. Then, an integral curve of $e_n$ is either a circle or line.
\end{Prop}
\begin{proof}
By using \eqref{Null2TypeSurfq1EqRESULT1c}, we get
\begin{eqnarray}\label{BihSurfCirclEq1}
\widetilde\nabla_{e_n}e_n&=-\omega_{1n}(e_n)e_1+k_nN, \quad \widetilde\nabla_{e_n}e_1=\omega_{1n}(e_n)e_n,\quad \widetilde\nabla_{e_n}N=-k_ne_n,
\end{eqnarray}
Moreover, \eqref{Null2TypeSurfq1EqRESULT1d} and \eqref{Null2TypeSurfq1EqRESULT1b} imply that $\omega_{1n}(e_n)$ and $k_n$ are constant on any (connected) integral curve $\alpha$ of $e_n$. Let  $t,n$ be tangent and normal vector fields of $\alpha$. Note that we have $t=\left. e_n\right|_\alpha.$ If 
$$\|\hat\nabla_t t\|=a=0$$ 
then $\alpha$ is a line and proof is completed, where $\hat\nabla$ is the Levi-Civita connection of $\alpha$ and $a$ is the constant given by $\left.\big(\omega_{1n}(e_n)^2+k_n^2\big)^{1/2}\right|_\alpha$. 

We assume $\hat\nabla_t t\neq0$. Then, we have $n=\hat\nabla_t t/\|\hat\nabla_t t\|$. From \eqref{BihSurfCirclEq1} we have
$\hat\nabla_tt=an,\quad \hat\nabla_t n=-an$.
Thus, $\alpha$ is planar and its curvature $a>0$.
\end{proof}

By summing up  \eqref{Null2TypeSurfq1EqRESULT1ALL}, we see that \eqref{HSurfaceCasep2PosVecEq3v2} is satisfied for $q=1.$ Thus, by taking into account Proposition \ref{BihSurfCircl}, we have the following proposition which can be proved like Theorem \ref{HSurfClassTheo}.
\begin{Prop}\label{HSurfClassTheoSub}
If there is a null 2-type hypersurface with shape operator given by \eqref{Null2TSurfShapeOp}, then it must be congruent to one of the following hypersurfaces.
\begin{enumerate}
 \item[(i)] A generalized rotational hypersurfaces given by \eqref{HSurfaceCasep2PosVecCase1} with $p=n-2$, $q=1$ for some functions satisfying $\psi'^2+\phi'^2=1$ and \eqref{HSurfaceCasep2PosVecCase1Cond},

\item[(ii)] A generalized cylinder over a rotational hypersurface, given by \eqref{HSurfaceCasep2PosVecCase2}  with $p=n-2$, $q=1$ for some functions satisfying $\psi'^2+\phi'^2=1$ and \eqref{HSurfaceCasep2PosVecCase2Cond},

\item[(iii)] A generalized cylinder over a rotational surface, given by \eqref{HSurfaceCasep2PosVecCase2}  with $p=1$, $q=n-2$, for some functions satisfying $\psi'^2+\phi'^2=1$ and \eqref{HSurfaceCasep2PosVecCase2Cond}.
\end{enumerate}
\end{Prop}

\end{document}